\documentclass[12pt]{article}

\usepackage{amssymb,amsmath,amsfonts,amssymb}
\usepackage{graphics,graphicx,color,hhline}
\usepackage{bbm} 
\usepackage{euscript}
\usepackage{mathtools}

\usepackage{authblk}

\def\@abssec#1{\vspace{.05in}\footnotesize \parindent .2in 
{\bf #1. }\ignorespaces} 

\graphicspath{{/EPSF/}{Figures/}}   

\newtheorem{theorem}{Theorem}[section]
\newtheorem{lemma}[theorem]{Lemma}

\def \Rm {\mathbb R}

\def \NN {\mathbb N}

\newcommand{\eps}{\varepsilon}

\newcommand{\be}{\begin{equation}}
\newcommand{\ee}{\end{equation}}
\newcommand{\bea}{\begin{eqnarray}}
\newcommand{\eea}{\end{eqnarray}}
\newcommand{\bee}{\begin{eqnarray*}}
\newcommand{\eee}{\end{eqnarray*}}

\newcommand{\pdr}[2]{\dfrac{\partial{#1}}{\partial{#2}}}

\newcommand{\bu}{\mathbf u} \newcommand{\bv}{\mathbf v}
\newcommand{\bw}{\mathbf w}

\newcommand{\bx}{\mathbf x} 
 
\newcommand{\bD}{\mathbf D} 
\newcommand{\bG}{\mathbf G} \newcommand{\bH}{\mathbf H}
\newcommand{\bW}{\mathbf W}
\newcommand{\bK}{\mathbf K} 
 
\newcommand{\bR}{\mathbf R} \newcommand{\bS}{\mathbf S}
\newcommand{\bV}{\mathbf V}
\newcommand{\bE}{\mathbf E}
\newcommand{\bB}{\mathbf B}

\def\fref#1{{\rm (\ref{#1})}}

\newcommand{\bU}{\mathbf U}

\newcommand{\cout}[1]{}

\begin{document}
\title{Instantaneous time mirrors and wave equations with  time-singular coefficients}
\author[,1]{Olivier Pinaud \footnote{pinaud@math.colostate.edu}}
 \affil[1]{Department of Mathematics, Colorado State University, Fort Collins CO, 80523}

 \maketitle

 \begin{abstract}

   We study in this work the concept of instantaneous time mirrors that were recently introduced in the physics literature in \cite{Fink-nphys}. Instantaneous time mirrors offer a new method for time reversal with a simplified experimental setup compared to classical techniques. At the mathematical level, instantaneous time mirrors are modeled by singularities in the time variable in the coefficients of a wave equation, and a prototype of such singularity is a Dirac delta. Our main goal in this work is to obtain refocusing estimates for the wavefield that quantify the quality of time reversal. This amounts to analyze the wave equation with Dirac-type singularities and develop a proper regularity theory as well as derive uniform estimates.
   
   \end{abstract}

\section{Introduction}

This work is concerned with the mathematical analysis of \textit{Instantaneous Time Mirrors (ITM)} that were introduced recently in \cite{Fink-nphys}, and which offer a new avenue for time reversal. The latter is a technique developed by M. Fink and collaborators in the nineties, see e.g. \cite{fink-PT}, based on the idea that if time is reversed in the wave equation $\partial_t^2 u=\partial_x^2u$  for instance, that is $t$ becomes $-t$, then the equation is not changed. This fact was successfully exploited in order to focus waves: suppose that (i) a signal is emitted from a point source at a time $t=0$ and propagates according to some time-reversible equations (e.g. a linear hyperbolic system without absorption such as acoustic, elastic, or electromagnetic wave equations), then (ii) is recorded and time-reversed at time $T$ (i.e. what is recorded last is sent back first), and finally (iii) is re-emitted for back-propagation during a time $T$. Then, at time $2T$, the signal refocuses at the location of the point source. The quality of refocusing  depends on various factors, such as how much of the signal was recorded during reversal and how heterogeneous is the underlying medium of propagation.

Time reversal has found many important applications in medical imaging, non-destructive testing, and telecommunications for instance. Its main practical difficulty is the measurement/reversal process: detectors (transducers) must operate both as recorders and emitters, and must occupy a sufficiently large domain of space for sharp focusing. This is often difficult to realize.

The groundbreaking nature of \cite{Fink-nphys} is that time reversal can actually be achieved without any measurements and a complex experimental apparatus. The main idea is that sudden and strong perturbations in the medium of propagation generate back-propagating waves that refocus at the emission point. The procedure is referred to in \cite{Fink-nphys} as creating an instantaneous time mirror, and opens interesting new perspectives as on the one hand the experimental procedure is simplified, and on the other some situations where time reversal was not feasible (e.g. quantum systems where phases are difficult to measure and the state of the system is modified by measurements) are amenable to reversal provided the background can be controlled. Note that ITM fall into the context of time refraction and time reflection where energy is in general not conserved, contrary to spatial refraction/reflection, see \cite{mendona}.

At the mathematical level, ITM are modeled by time-singular coefficients in hyperbolic equations. The prototype of such singularity is a delta function at a given time $T$, that represents the (strong) perturbation due to the ITM at $T$. In \cite{Fink-nphys}, this is the sudden and strong shaking of a water tank that changes abruptly the wave velocity of surface waves. A signal emitted at time $t=0$, perturbed by an ITM at $t=T$, will then refocus at its source at time $t=2T$.

The objective of the present work is to continue the analysis of ITM that we began in \cite{BFP}. In the latter, we analyzed the refocusing wave in the context of wave equations with Dirac-type singularities. We proved refocusing estimates (the notion will be introduced further in the paper) for the wave equation with \textit{spatially constant coefficients}. This allowed us to use the Fourier transform and pursue a fine analysis of the time-singularities in Fourier space. We generalize in this work the refocusing estimates to wave equations with smooth variable coefficients. Naturally, the use of Fourier techniques is not possible and we need to resort to different methods. A feature of ITM is that the refocusing wave is the time derivative of the original wave. This will be reflected in the estimates where there is a loss of a derivative.

The article is organized as follows: in Section \ref{main}, we introduce some background on ITM and state our main results. The latter consists of two theorems: in the first one, we obtain refocusing estimates for the wave equation with varying coefficients when the ITM is modeled by an approximation of a Dirac delta; in the second theorem, we show that the system obtained by removing the approximation is well-posed. The proofs of these theorems are given in Sections \ref{secproof1} and \ref{secproof2}. We provide in the Appendix derivations of wave equations with time-dependent coefficients in the context of electromagnetics, elasticity, and fluids.

\paragraph{Acknowledgment.} This work is supported by NSF CAREER Grant DMS-1452349 and NSF grant DMS-2006416. The author would like to thank Lenya Ryzhik for an interesting discussion.

\section{Main results} \label{main}

We start by introducing some background on ITM.

\subsection{Preliminaries}

\paragraph{Notations.} For $d\geq 1$, $p \in [1,\infty]$ and $s \in \Rm$, we denote by $L^p(\Rm^d)$ and $H^s(\Rm^d)$ the usual Lebesgue and Sobolev spaces. $BV(\Omega)$ is the space of functions with bounded variations in $\Omega$, with $BV_{\rm{loc}}$ and $L^p_{\rm{loc}}$ the local versions of $BV$ and $L^p$. $C^\infty(\Rm^d)$ is the space of infinitely differentiable functions, and $C^\infty_c(\Rm^d)$ the space of $C^\infty(\Rm^d)$ functions with compact support. For two real-valued functions $f$ and $g$, we define $(f,g)=\int_{\Rm^d} f(x) g(x) dx$.

\paragraph{The wave equation.} We consider the following wave equation (we suppose here all variables have been  non-dimensionalized):

\be \label{WE}
\partial_t^2 u_\eps  = a(x) \nabla \cdot \Big(b(x)\big(1+\chi(x)\eta_\eps(t)\big) \nabla u_\eps \Big), \qquad (t,x)\quad \textrm{on} \quad \Rm_+ \times \Rm^d,
\ee
where $a$ and $b$ are two functions modeling a smooth, unperturbed background, and the term $\chi(x)\eta_\eps(t)$ models the action of the ITM. For simplicity of the analysis, we suppose that \fref{WE} is posed over the entire $\Rm^d$, $d \geq 1$. Our results can be  generalized to bounded domains with appropriate boundary conditions without difficulty. For $T>0$ and $0<\eps \leq T$ given, the function $\eta_\eps$ has the form
$$
\eta_\eps(t)=
\left\{
  \begin{array}{ll}
    \eta_0 \eps^{-1} &\textrm{when}\quad |t-T|<\eps/2\\
    0 & \textrm{otherwise}.
\end{array}
 \right.
$$
The ITM acts therefore at time $t=T$ over a window $\eps$ and with amplitude $\eta_0 \eps^{-1}$. The function $\eta_\eps$ is an approximation of a delta function at $t=T$ with weight $\eta_0$, and $\chi$ is a regularized version of the characteristic function of the spatial domain where the ITM acts (we can have $\chi=1$ if the ITM acts on the whole $\Rm^d$). 
Our analysis can be straightforwardly generalized to multiple ITM perturbations of the form
$$
\sum_{i=1}^N \chi_i(x) \eta_\eps^{(i)}(t),
$$
where $\eta_\eps^{(i)}$ has the same form as $\eta_\eps$ and is centered at $T_i$.

We will suppose that $a$ and $b$ are in $C^\infty(\Rm^d)$ with bounded derivatives and positive constants $\underline{a}$, $\bar{a}$, $\underline{b}$, $\bar{b}$ such that
\be \label{ab}
\underline{a} \leq a(x)\leq \bar{a}, \qquad \underline{b}\leq b(x)\leq \bar{b}.
\ee
We assume as well that $\chi \in C^\infty(\Rm^d)$ with bounded derivatives, and that $\chi$ is positive to ensure that the velocity $c^2_\eps(t,x)=a(x)b(x)(1+\chi(x)\eta_\eps(t))$ remains positive. All the $C^\infty$ regularity assumptions are not crucial and can be relaxed to a control of just a few derivatives.

The wave equation equation \fref{WE} is complemented with the initial conditions
\be \label{IC}
u(t=0,x)=u_0(x) \in H^3(\Rm^d), \qquad \partial_t u(t=0,x)=u_1(x) \in H^2(\Rm^d). 
\ee
As will be proved further, such a regularity is needed in order to obtain optimal refocusing estimates. We suppose that $u_0$ and $u_1$ are real-valued, and as a consequence the solution $u$ is real-valued as well since the coefficients in \fref{WE} are real. The case of complex-valued initial conditions is simply obtained by separating real and imaginary parts.

Physical derivations of wave equations with time-dependent coefficients of the form \fref{WE} are given in the Appendix in the context of electromagnetics, elasticity, and surface waves. The solution $u_\eps$ is then the surface height for the latter, or related to the magnetic, electric, displacement field for the former.

\paragraph{Time reversal and refocusing.} We introduce in this paragraph some necessary background on ITM and some results from \cite{BFP}. In order to investigate the refocusing induced by the ITM, it is convenient to recast \fref{WE} as a first-order system: let $\bv_\eps \in \Rm^d$ such that
\be \label{AWE}
  \big[b (1+\chi \eta_\eps)\big]^{-1} \frac{\partial \bv_\eps }{\partial t} + \nabla u_\eps=0,\qquad a^{-1} \pdr {u_\eps} t+\nabla\cdot \bv_\eps=0.
  \ee
The system \fref{AWE} can itself be stated as
  \be \label{HS}
A(x) \frac{\partial \bu_\eps }{\partial t} + D^j \frac{\partial \bu_\eps}{\partial x_j}=V_\eps(t,x) \bS[\bu_\eps],
\ee
where $V_\eps(t,x)=b(x) (1+\chi(x) \eta_\eps(t))$, and $\bu_\eps=(\bv_\eps,u_\eps)$, $\bS[\bu_\eps]=-(\nabla u_\eps,0) $ (which are both considered as column vectors), $A=\textrm{Diag}(b^{-1},\cdots,b^{-1},a^{-1})$ ($b^{-1}$ repeated $d$ times), and $(D^j)_{mn}= \delta_{m, (d+1)} \delta_{n,j}+\delta_{n, (d+1)} \delta_{m,j}$, with $j=1,\cdots,d$, and $m,n=1,\cdots,d+1$. Here and below, we use the summation convention over repeated indices. Equation \fref{HS} is equipped with the initial condition $\bu_0=(\bv_0,u_0)$, where $\bv_0$ is such that $a^{-1} u_1+ \nabla \cdot \bv_0=0$. We will need the Green's function of the unperturbed equation \fref{HS},  i.e. with $V=0$, defined by

\be \label{GF}
A(x) \frac{\partial \bG(t,x,y) }{\partial t} + D^j \frac{\partial \bG(t,x,y)}{\partial x_j}=0, \qquad \bG(0,x,y)= I \delta(x-y) ,
\ee
where $I$ is the $(d+1)\times (d+1)$ identity matrix and $\delta$ the Dirac delta. We extend the Green's function $\bG_t$ to negative values of $t$ by solving \fref{GF} for $t<0$, and find $\bG_{-t} \Gamma=\Gamma \bG_{t}$, for $\Gamma=\textrm{Diag}(-1,\cdots,-1,1)$ the time reversal matrix.
We will use the notation
$$
\bG_t(\bu)(x)=\int_{\Rm^d} \bG(t,x,y) \bu(y) dy.
$$

The unperturbed solution to \fref{HS}, i.e. with $V=0$, is denoted by $\bU=(\bV,U)$. The following decomposition of $\bu_\eps$, obtained in \cite{BFP}, is key to understanding the refocusing effects of an ITM: the wavefield $\bu_\eps$ can be written as
\be \label{declem}
\bu_\eps(t)=\bU(t)+\bu_R(t)+\bu_F(t)+\bR_\eps(t),
\ee
where $\bu_R$ is a backward (time-reversed) propagating wave, $\bu_F$ a forward propagating wave, and $\bR_\eps$ a correction term. The precise expression of these terms can be found in \cite{BFP}, and  $\bR_\eps$ will be our main focus. The decomposition \fref{declem} holds for any solution to the wave equation, and in general $\bR_\eps$ has no reason to be small at all and can actually dominate. The latter reads

$$
\bR_\eps(t)=\int_0^t \bG_{t-s}\Big(V_\eps(s) A^{-1}(\bS[\bu_\eps](s)-\bS[\bU](s))\Big) ds.
$$

In the ITM context, the term $\bR_\eps$ becomes negligible for small $\eps$, and this is the mathematical explanation for the observation of a time-reversed focusing wave. The time-reversed field $\bu_R$ at the refocusing time $t=2T$ is then the dominating term and reads
\bee
\bu_R(2T,x)&=&\int_{\Rm^d} \bK_\eps(x,y) \partial_t \bU(y)dy
\eee
with
\bee
\bK_\eps(x,y)&=&-\frac{\eta_0}{2\eps}\int_{-\frac{\eps}{2}}^{\frac{\eps}{2}}  \Gamma \bG(2s,x,y) ds.
\eee
The kernel $\bK_\eps$ is an approximation of $-\delta(x-y) \Gamma /2 $ when $\eps$ is sufficiently small since $\bG(0,x,y)=\delta(x-y) I$. Some blurring in the refocusing is introduced when $\eps$ is not zero, and refocusing is perfect is the limit $\eps \to 0$. An important observation is that one reconstructs the time derivative of the initial condition $\partial_t \bU(t=0)$ and not the initial condition $\bU(t=0)$. The ITM hence acts as a time differentiator, and this will be seen in the estimates.

The fact that $\bR_\eps$ is negligible when $\eps$ is small is proved in \cite{BFP} when the coefficients $a$, $b$ and $\chi$ are all constant. This allows for the use of the Fourier transform and \fref{WE} is reduced in Fourier space to

$$
\partial_t^2 \hat u_\eps(t,\xi)+\eta_\eps(t) |\xi|^2 \hat u_\eps(t,\xi).
$$
This is a one-dimensional Schr\"odinger equation in the variable $t$ with the singular potential $-\eta_\eps(t) |\xi|^2$. One can then precisely estimate $\hat u_\eps$ in terms of $\eps$ and $\xi$ and obtain sufficient regularity to treat $\bR_\eps$.

Our goal in this work is to generalize this result to the case of variable $a$, $b$ and $\chi$. The fact that $\bR_\eps(t)$ is negligible when $t \geq T+\eps/2$ (the time right after the pertubation) follows from the following heuristical arguments: the non-zero component of $\bS[\bU]-\bS[\bu_\eps]$ in the definition of $\bR_\eps$ reads, for any $s \in [T-\frac{\eps}{2},T+\frac{\eps}{2} ]$,
$$
\nabla (u_\eps(s)-U(s))=\nabla (u_\eps(s)-u_\eps(s-\eps))+\nabla (u_\eps(s-\eps)-U(s)).
$$
As $s-\eps < T-\frac{\eps}{2}$, we are therefore before the perturbation, and the second term in the r.h.s. is actually equal to $\nabla U(s-\eps)-\nabla U$. Since $\nabla U$ is smooth (i.e. it has at least one time derivative), this second term is negligible provided $\eps$ is small compared to a parameter estimating $\nabla U$ in some norm. The same applies to the first term provided $\nabla u_\eps$ has some regularity in time. The essential ingredient in estimating $\bR_\eps$ is then a uniform bound for $\partial_t \nabla u_\eps$ in $L^\infty_{\rm{loc}}(\Rm_+,L^2(\Rm^d))$, which will eventually provide us with an optimal control in terms of $\eps$, that is $\bR_\eps=O(\eps)$ in some appropriate sense. This bound is of course more difficult to obtain than in the constant coefficients case since the Fourier transform is not available. We refer to the relation $\bR_\eps=O(\eps)$ as an optimal refocusing estimate. It is possible to obtain non optimal estimates of the form $\bR_\eps=O(\eps^\gamma)$ for $\gamma<1$ assuming less regularity on the data.\\

We state now our main results.

\subsection{Results}

Using standard methods and assumptions \fref{ab}-\fref{IC}, see e.g. \cite{evans}, Chapter 7, it is not difficult to establish that \fref{WE} admits a unique solution $u_\eps$ such that $u_\eps \in L^\infty_{\rm{loc}}(\Rm_+,H^3(\Rm^d))$ and $\partial_t^2 u_\eps \in L^\infty_{\rm{loc}}(\Rm_+,H^1(\Rm^d))$. Our first result is the theorem below, that provides us with a uniform estimate on $u_\eps$. This estimate is used in the second statement of the theorem in order to obtain the optimal refocusing estimate $\bR^\eps=O(\eps)$.

\begin{theorem} \label{mainth}Let $u_\eps $ be the solution to \fref{WE} with the initial conditions given in \fref{IC}. Then, we have the estimate, for all $\tau>0$,
\be \label{unif}
\| u_\eps\|_{L^\infty(0,\tau,H^2)}+\| \partial_t u_\eps\|_{L^\infty(0,\tau,H^1)} \leq C \| u_0\|_{H^3}+ C \| u_1\|_{H^2},
\ee
where $C\equiv C(a,b,\chi,\eta_0,\tau)$ is independent of $\eps$. Moreover, write $\bR_\eps=(\bW_\eps,w_\eps)$, where $\bW_\eps$ is a vector with $d$ components. We have then the estimate, for all $\tau>0$,
\be \label{unif2}
\|\bW_\eps\|_{L^\infty(0,\tau,H^{-1})}+\|w_\eps\|_{L^\infty(0,\tau,L^2)} \leq C \eps,
\ee
where $C\equiv C(a,b,\chi,\eta_0,\tau,u_0,u_1)$ is independent of $\eps$.
\end{theorem}

Note the loss of a spatial derivative in the uniform estimates \fref{unif}: one needs $u_\eps(t=0) \in H^3$ and $\partial_t u_\eps(t=0) \in H^2$ to obtain a uniform control of $u_\eps(t)$ and $\partial_t u_\eps(t)$ in $H^2$ and $H^1$, respectively. This is induced by the time singularity of the coefficients created by the ITM. This loss is optimal in the sense it is also there in the case of constant coefficients addressed in \cite{BFP} where the Fourier transform allows for exact calculations. Owing to the heuristics that a time derivative is equivalent to a spatial derivative for the free the wave equation, this spatial loss can be related to the fact that the ITM acts as a time differentiator, as mentioned earlier.

The proof of \fref{unif} is based on three estimates obtained in six steps, and our main goal is to get the bound on $\partial_t \nabla u_\eps$ in $L^\infty_{\rm{loc}}(\Rm_+,L^2(\Rm^d))$. In the first step, we derive a classical energy estimates for $\nabla u_\eps$ in $L^2$; now, this estimate is only uniform in $\eps$ at the location of the perturbation, modeled by $\chi$. In the second step, we extend the estimate to $\Rm^d$ at the price of losing one derivative. At the end of these first two steps, we have obtained uniform estimates for $u_\eps$ in $L^\infty_{\rm{loc}}(\Rm_+,L^2(\Rm^d))$ and for $\partial_t u_\eps$ in $L^\infty_{\rm{loc}}(\Rm_+,H^{-1}(\Rm^d))$. We then differentiate the equation twice, and use the same procedure as in steps 1 and 2 to arrive at the desired result for $\partial_t \nabla u_\eps$.

Estimate \fref{unif2} is a direct consequence of \fref{unif} and the equations satisfied by the different components of $\bR_\eps$.

Our second result concerns the limit $\eps \to 0$, and provides us with an existence and uniqueness theorem for the limiting wave equation with a Dirac delta at time $t=T$. The ITM is seen via a jump condition on the time derivative of the solution at $t=T$.

\begin{theorem} \label{mainth2}
  Consider the wave equation, for $t>0$, and $t \neq T$,
  \be \label{WE2}
\partial_t^2 u  = a(x) \nabla \cdot \big(b(x) \nabla u \big),
\ee
equipped with the initial conditions \fref{IC} and the following jump condition at $t=T$:
\be \label{jump}
\partial_t u(T^+,x)=\partial_t u(T^-,x)+ \eta_0 a(x) \nabla \cdot \Big(b(x) \chi(x) \nabla u(T,x) \Big).
\ee
Then, the above system admits a unique solution $u$ in $C^0(\Rm_+,H^1(\Rm^d)) \cap L_{\rm{loc}}^\infty(\Rm_+,H^{2}(\Rm^d))$ such that $\partial_t u \in L^\infty_{\rm{loc}}(\Rm_+,H^{1}(\Rm^d)) \cap BV_{\rm{loc}}(\Rm_+,L^2(\Rm^d))$. Equation \fref{WE2} is verified almost everywhere in $(0,T) \times \Rm^d$ and in $(T,+\infty) \times \Rm^d$, and the jump condition \fref{jump} is satisfied in $H^{-1}(\Rm^d)$.

\noindent Moreover, $u_\eps$ converges to $u$ as $\eps \to 0$ in $ L_{\rm{loc}}^\infty(\Rm_+,H^{2}(\Rm^d))$ weak-$*$.
  \end{theorem}

The rest of the paper is dedicated to the proofs of Theorems \ref{mainth} and \ref{mainth2}.

\section{Proof of Theorem \ref{mainth}} \label{secproof1}
We will consistenly use that $a,b$ and $\chi$ are smooth functions, and that $a,b$ are positive and bounded below. We will not recall these facts for each estimate, and for simplicity will not make explicit the dependency  of the various constants on $a,b,\chi$. The estimates will be derived for regular initial conditions $u_0$ and $u_1$ in $C_c^\infty(\Rm^d)$ to justify the calculations, in particular the integration by parts over $\Rm^d$ using finite speed of propagation of the support. Hence,  we will work with a solution $u_\eps$ that is infinitely differentiable with respect to the spatial variables and with bounded support for finite times, and that has two bounded derivatives with respect to $t$. The case $u_0 \in H^3(\Rm^d)$ and $u_1 \in H^2(\Rm^d)$ follows by a simple limiting argument.

\subsection{First estimate}

We derive in this section uniform bounds for $u_\eps$ and $\partial_t u_\eps$ in $L^2(\Rm^d)$ and $H^{-1}(\Rm^d)$, respectively. The first step is to obtain a uniform control at the location of the perturbation. 

\paragraph{Step 1:} We begin with a classical energy estimate. Let $$
E_\eps(t)=\frac{1}{2}\int_{\Rm^d} \big( |\partial_t u_\eps(t,x)|^2+ c^2(x) |\nabla u_\eps(t,x)|^2\big) dx,
$$
and 
$$
F_\eps(t)=\frac{1}{2}\int_{\Rm^d} c^2(x) \chi(x) |\nabla u_\eps(t,x)|^2 dx,
$$
where we have set $c^2(x)=a(x)b(x)$. All calculations in the proof are justified since $u_\eps$ has sufficient regularity. In particular, the wave equation \fref{WE} is satisfied everywhere on $\Rm_+ \times \Rm^d$, and we find, by multiplying \fref{WE} by $\partial_t u_\eps$ and integrating in $x$, for all $t>0$,
\bee
\frac{d E_\eps(t)}{ dt}+\eta_\eps(t)\frac{d F_\eps(t)}{ dt}&=&-\int_{\Rm^d} b(x)(1+\chi(x) \eta_\eps(t))\nabla a(x) \cdot \nabla u_\eps(t,x) \, \partial_t u_\eps(t,x) dx\\
&=:&A_\eps(t).
\eee
We have, using the Cauchy-Schwarz inequality, for all $t>0$,
$$
|A_\eps(t)| \leq C(1+\eta_\eps(t)) E_\eps(t).
$$
Let now $t_\eps^\pm= T \pm \frac{1}{2}\eps $, and $\eta_{\eps,0}=\eta_0 /\eps$.
Then, since $\eta_\eps(t)= \eta_{\eps,0}$ for $t\in (t_\eps^-,t_\eps^+)$,
\bee
\frac{d E_\eps(t)}{ dt}+\eta_{\eps,0}\frac{d F_\eps(t)}{ dt}\leq C (1+\eta_{\eps,0}) E_\eps(t), \qquad t \in (t_\eps^-,t_\eps^+),
\eee
which yields, for all $t\in (t_\eps^-,t_\eps^+)$,
\be \label{E1}
E_\eps(t)+\eta_{\eps,0}F_\eps(t) \leq \big(E_\eps(t_\eps^-)+\eta_{\eps,0}F_\eps(t_\eps^-)\big)e^{C (\eps+\eta_0)}.
\ee
As a consequence, for all $t\in (t_\eps^-,t_\eps^+)$,
\be \label{estF}
F_\eps(t) \leq \big( \eps \eta_0^{-1} E_\eps(t_\eps^-)+F_\eps(t_\eps^-)\big)e^{C (\eps+\eta_0)}=:(M_{0,\eps})^2.
\ee
This last estimate is uniform in $\eps$ since $E_\eps(t)$ and $F_\eps(t)$ are continuous in time and independent  of  $\eps$ before the perturbation. Note that \fref{estF} only provides us with a uniform control at the location of the perturbation, and that estimate  \fref{E1} does  not yield a uniform bound over the entire $\Rm^d$. For this, we need to exploit the latter estimate on $F_\eps$ and go back  to the wave equation, at the price  of losing one derivative.
\paragraph{Step  2.} We exploit here estimate \fref{estF} on $F_\eps$ to control $u_\eps$ and $\partial_t u_\eps$ over $\Rm^d$. We rewrite \fref{WE} as
$$
\partial_t^2 u_\eps = -a  L_0 u_\eps+ a  u_\eps +a \eta_\eps L_1  u_\eps  \qquad (t,x)\in \Rm_+ \times \Rm^d,
$$
where $$
L_0 u=-\nabla \cdot \big(b  \nabla u \big)+u, \qquad
L_1 u= \nabla \cdot \big( b \chi \nabla u \big).
$$
Recalling that $b \in C^\infty(\Rm^d)$ with $0<\underline{b}\leq b(x) \leq \bar{b}$ and bounded derivatives, the operator $L_0$ is self-adjoint when equipped with the domain $H^2(\Rm^d)$. For $s \in \Rm$,  the inverse operator $L_0^{-1}$ is an isomorphism from $H^{s-2}(\Rm^d)$ to $H^{s}(\Rm^d)$, and its square root an isomorphism from $H^{s-1}(\Rm^d)$ to $H^{s}(\Rm^d)$, We will use several times the following lemma.
\begin{lemma} \label{extend}
  Let $\tau>0$, $f \in L^\infty(0,\tau,L^2(\Rm^d))$, and $u \in L^{\infty}(0,\tau,H^2(\Rm^d))$ with $\partial^2_{t} u \in L^\infty(0,\tau,L^2(\Rm^d))$ such that
  \be \label{eqL0}
\partial_t^2 u = -a L_0  u + a u+f, \qquad (t,x)\; a.e.
\ee
For all $t \in (0,\tau)$, we have then the estimate, for some $C=C(a,b,\tau)>0$ independent of $u$ and $f$:
\begin{align*}
  \|u(t)\|_{L^2}+\|&L_0^{-1/2} (a^{-1} \partial_t u(t)) \|_{L^2} \\
  &\leq C \|u(0)\|_{L^2}+C\|L_0^{-1/2} (a^{-1} \partial_t u(0)) \|_{L^2}
+C \int_0^t \|a^{-1} f(s)\|_{H^{-1}} ds.
\end{align*}
\end{lemma}

  \begin{proof}
    Consider first the weighted $L^2$ norm,
    $$
    \|u\|^2_a=\int_{\Rm^d} |u(x)|^2 a^{-1}(x) dx,
    $$
    which is equivalent to the usual $L^2$ norm since $a>0$ is bounded below and above. The calculations below are justified since $u$ has the required regularity. We have
    \bee
    \frac{1}{2} \frac{d  }{ dt} \|u(t)\|^2_a &=& (a^{-1} \partial_t u, u)\\
    &=&- (a^{-1} \partial_t u, L_0^{-1} a^{-1} \partial^2_{t}u)+ (a^{-1} \partial_t u, L_0^{-1} u)+ (a^{-1} \partial_t u, L_0^{-1} a^{-1} f),
    \eee
    where we used \fref{eqL0} to express $u$. With
    $$
G(t)=\frac{1}{2} \left( \|u(t)\|^2_a+\|L_0^{-1/2} (a^{-1} \partial_t u) \|^2_{L^2}\right),
$$
we find
$$
\frac{d G(t)}{ dt}=B(t):= \left(L_0^{-1/2} (a^{-1} \partial_t u), L_0^{-1/2} (a^{-1} f)\right)+ \left(L_0^{-1/2} a^{-1} \partial_t u, L_0^{-1/2} u\right).
$$
We have, using the Cauchy-Schwarz inequality,
\bee
|B(t)| &\leq& \|L_0^{-1/2} (a^{-1} \partial_t u)\|_{L^2} \big(\|L_0^{-1/2} a^{-1} f\|_{L^2}+C\| u \|_{L^2} \big)\\
&\leq & CG(t)+C  G^{1/2}(t) \|a^{-1} f\|_{H^{-1}},
\eee
where we exploited the fact that $L_0^{-1/2}$ is an isomorphism from $H^{-1}(\Rm^d)$ to $L^2(\Rm^d)$. Then,
$$
\frac{d G(t)}{ dt} \leq CG(t)+C G^{1/2}(t) \|a^{-1} f(t)\|_{H^{-1}}, \qquad t. \; a.e. 
$$
which is equivalent to
$$
\frac{d}{ dt} \left( e^{-Ct} G(t)\right) \leq C e^{-Ct} G^{1/2}(t) \|a^{-1} f(t)\|_{H^{-1}}, \qquad t. \; a.e. 
$$
This yields the estimate, for $t \in (0,\tau)$,
$$
G^{1/2}(t)\leq C G^{1/2}(0)+C \int_{0}^t \|a^{-1} f(s)\|_{H^{-1}} ds,
$$
which concludes the proof after direct algebra.
\end{proof}

\medskip

Using the previous lemma with $u(t)=u_\eps(t_\eps^-+t)$, $f(t)=a \eta_\eps(t_\eps^-+t) L_1  u_\eps(t_\eps^-+t)$, we find the estimate, for $t \in (t_\eps^-, t_\eps^+)$,
\begin{align*}
  \|u_\eps(t)\|_{L^2}+\|&\partial_t u_\eps(t)\|_{H^{-1}} \\
  &\leq C \|u_\eps(t_\eps^-)\|_{L^2}+C\|\partial_t u_\eps(t_\eps^-)\|_{H^{-1}}+ C \int_{t_\eps^-}^t \eta_\eps(s) \| b \chi \nabla u_\eps(s)\|_{L^2} ds.
\end{align*}
Together with  \fref{estF}, it follows, for all $t\in (t_\eps^-,t_\eps^+)$,
\be \label{M1}
 \|u_\eps(t)\|_{L^2}+\|\partial_t u_\eps(t)\|_{H^{-1}} \leq C \|u_\eps(t_\eps^-)\|_{L^2}+C \|\partial_t u_\eps(t_\eps^-)\|_{H^{-1}} + C M_{0,\eps}=:M_{1,\eps}.
\ee

This is our first uniform estimate over $\Rm^d$. Note that in order to control $u_\eps$ and $\partial_t u_\eps$ in $L^2(\Rm^d)$ and $H^1(\Rm^d)$, respectively, we need one more derivative for each via the constant $M_{0,\eps}$. We iterate now twice in order to control higher spatial derivatives of $u_\eps$ and $\partial_t u_\eps$. The method is similar as above.
\subsection{Second estimate}

\paragraph{Step 1.} We differentiate \fref{WE} with respect to $x_j$, $j=1,\cdots,d$, and introduce $v^{(j)}:=\partial_{x_j} u_\eps$.  We use the shorthand $f_j:=\partial_{x_j} f$ to denote the partial derivative of a function $f$ with respect to $x_j$. With the previous definitions of $L_0$ and $L_1$ at hand, we have
\be \label{eqvj}
\partial_t^2 v^{(j)} = -a L_0 v^{(j)}+ a  v^{(j)} +a \eta_\eps  L_1  v^{(j)}+L_2 u_\eps+ L_3 u_\eps, 
\ee
where
\bee
L_2 u_\eps&=&  ( c^2+c^2 \eta_{\eps} \chi )_j \Delta u_\eps\\
L_3 u_\eps&=&a \nabla (b+b \eta_\eps \chi)_j \cdot \nabla u_\eps+a_j \nabla (b+b \eta_\eps \chi) \cdot \nabla u_\eps.
\eee
With
$$
E^{(j)}_\eps(t)=\frac{1}{2}\int_{\Rm^d} \big( |\partial_t v^{(j)}(t,x)|^2+ c^2(x) |\nabla v^{(j)}(t,x)|^2\big) dx,
$$
and 
$$
F^{(j)}_\eps(t)=\frac{1}{2}\int_{\Rm^d} c^2(x) \chi(x) |\nabla v^{(j)}(t,x)|^2 dx,
$$
we obtain from \fref{eqvj} the standard energy estimate, for $t>0$,
\bea \label{stE}
\frac{d E^{(j)}_\eps(t)}{ dt}+\eta_\eps(t)\frac{d F^{(j)}_\eps(t)}{ dt}&=&\big(\partial_t v^{(j)}, L_2 u_\eps+ L_3 u_\eps\big).
\eea
We estimate now the right-hand side. Classical interpolation yields first
$$
 \|L_2 u_\eps+L_3 u_\eps\|_{L^2} \leq C (1+\eta_\eps)(\|u_\eps\|_{L^2}+\|\Delta u_\eps\|_{L^2}).
 $$
 Moreover, there exists a constant $C>0$ independent of $u_\eps$ such that
 $$
 \|\Delta u_\eps(t)\|^2_{L^2} \leq C \sum_{j=1}^d \int_{\Rm^d} c^2(x) |\nabla v^{(j)}(t,x)|^2 dx.
 $$
 Introducing
 $$
 E_{1,\eps}(t):=\sum_{j=1}^d E^{(j)}_\eps(t), \qquad F_{1,\eps}(t):=\sum_{j=1}^d F^{(j)}_\eps,
 $$
 we then find, together with \fref{M1} in order to control the $L^2$ norm of $u_\eps$, for $t\in (t_\eps^-,t_\eps^+)$,
$$
 \|L_2 u_\eps(t)+L_3 u_\eps(t)\|_{L^2} \leq C (1+\eta_\eps)(M_{1,\eps}+(E_{1,\eps}(t))^{1/2}).
 $$
 Going back to \fref{stE}, we find for $t\in (t_\eps^-,t_\eps^+)$, after using the Young inequality,
\bee
\frac{d E_{1,\eps}(t)}{ dt}+\eta_\eps(t)\frac{d F_{1,\eps}(t)}{ dt}&\leq & C (1+\eta_\eps(t)) E_{1,\eps}+C (1+\eta_\eps(t)) M_{1,\eps}^2.
\eee
This provides us with the estimate, for $t\in (t_\eps^-,t_\eps^+)$,
$$
E_{1,\eps}(t)+\eta_{\eps,0}F_{1,\eps}(t) \leq C \big( E_{1,\eps}(t_\eps^-)+\eta_{\eps,0}F_{1,\eps}(t_\eps^-)+M_{1,\eps}^2\big),
$$
leading to
\be \label{estF1}
F_{1,\eps}(t) \leq C \big( \eps \eta_0^{-1}E_{1,\eps}(t_\eps^-)+F_{1,\eps}(t_\eps^-)+ \eps \eta_0^{-1} M_{1,\eps}^2\big)=:(M_{3,\eps})^2.
\ee
Again, the uniform estimate only holds at this point at the location of the perturbation, and it is extended below to $\Rm^d$ by using Lemma \ref{extend}.

\paragraph{Step 2.} We have first the following estimate, that is a consequence of \fref{estF1},
$$
\| L_1 v^{(j)} \|_{H^{-1}} \leq C \| b \chi \nabla v^{(j)}\|_{L^2} \leq C M_{3,\eps}. 
$$
The next two estimates are straightforward: 
\bee
\| L_3 u_\eps \|_{H^{-1}} &\leq& C(1+\eta_\eps) \| u_\eps \|_{L^2} \leq C(1+\eta_\eps) M_{1,\eps}\\
\| L_2 u_\eps \|_{H^{-1}} &\leq& C (1+\eta_\eps)\| \nabla u_\eps\|_{L^2}. 
\eee
Using now Lemma \ref{extend} for \fref{eqvj}, and summing from $j=1$ to  $d$, we find for $t\in (t_\eps^-,t_\eps^+)$:
\begin{align*}
  \|\nabla u_\eps(t)\|_{L^2}+\|\partial_t \nabla u_\eps(t)\|_{H^{-1}} \leq & \quad C \|\nabla u_\eps(t_\eps^-)\|_{L^2}+ C \|\partial_t \nabla u_\eps(t_\eps^-)\|_{H^{-1}}\\
  &+ C \int_{t_\eps^-}^t (1+\eta_\eps(s)) (M_{1,\eps}+M_{3,\eps}+\| \nabla u_\eps(s) \|_{L^2} )ds.
\end{align*}
Gronwall's  Lemma then yields the estimate, for $t\in (t_\eps^-,t_\eps^+)$:
\bea \nonumber
\|\nabla u_\eps(t)\|_{L^2}+\|\partial_t \nabla u_\eps(t)\|_{H^{-1}} &\leq& C \big(\|\nabla u_\eps(t_\eps^-)\|_{L^2}+\|\partial_t \nabla u_\eps(t_\eps^-)\|_{H^{-1}}+ M_{1,\eps}+M_{3,\eps} \big)\\
&=:&M_{4,\eps}. \label{Em4}
\eea

We have therefore just obtained a uniform  bound in $H^1(\Rm^d)$ and $L^2(\Rm^d)$ for $u_\eps$ and $\partial_t u_\eps$, respectively.  We iterate one last time to control  $\partial_t u_\eps$ in $H^1(\Rm^d)$, which is what is needed to prove refocusing. We do not detail some of the calculations since the method is similar to what is done is the previous steps.

\subsection{Third estimate}
\paragraph{Step 1.}
We differentiate \fref{eqvj} with respect to $x_i$, $i=1,\cdots,d$, and denote $v^{(ij)}:=\partial^2_{x_i x_j} u_\eps$. We have
\be \label{eqvij}
\partial_t^2 v^{(ij)} = -a L_0 v^{(ij)}+ a v^{(ij)} +a\eta_\eps L_1  v^{(ij)}+L_2 v^{(i)}+ L_3 v^{(i)}+L_2' v^{(j)}+ L_3' v^{(j)}+L_4 u_\eps+L_5 u_\eps, 
\ee
where
\bee
L_2' u&=&  ( c^2+c^2 \eta_{\eps} \chi )_i \Delta u\\
L_3' u&=&a \nabla (b+b \eta_\eps \chi)_i \cdot \nabla u+a_i \nabla (b+b \eta_\eps \chi) \cdot \nabla u\\
L_4 u&=& ( c^2+c^2 \eta_{\eps} \chi )_{ij} \Delta u\\
L_5 u&=& a_i \nabla (b+b \eta_\eps \chi)_{j} \cdot \nabla u+a \nabla (b+b \eta_\eps \chi)_{ij} \cdot \nabla u+a_{ij} \nabla (b+b \eta_\eps \chi) \cdot \nabla u.
\eee
With
\bee
E^{(ij)}_\eps(t)&=&\frac{1}{2}\int_{\Rm^d} \big( |\partial_t v^{(ij)}(t,x)|^2+ c^2(x) |\nabla v^{(ij)}(t,x)|^2\big) dx\\
F^{(ij)}_\eps(t)&=&\frac{1}{2}\int_{\Rm^d} c^2(x) \chi(x) |\nabla v^{(ij)}(t,x)|^2 dx,
\eee
we obtain from \fref{eqvij} the energy estimate
\be \label{eneg3}
\frac{d E^{(ij)}_\eps(t)}{ dt}+\eta_\eps(t)\frac{d F^{(ij)}_\eps(t)}{ dt}=\left(\partial_t v^{(ij)}, S_\eps\right),
\ee
where
$$
S_\eps=L_2 v^{(i)}+ L_3 v^{(i)}+L_2' v^{(j)}+ L_3' v^{(j)}+L_4 u_\eps+L_5 u_\eps.
$$
We now estimate $R_\eps$. Let
$$
\|D^3 u_\eps\|^2_{L^2}:= \sum_{i,j,k=1}^d \int_{\Rm^d} |\partial^3_{x_i x_j x_k} u_\eps|^2 dx.
$$
Direct calculations yield
\be \label{Reps}
 \|S_\eps\|_{L^2} \leq C (1+\eta_\eps)(\|u_\eps\|_{L^2}+\| D^3 u_\eps\|_{L^2}).
 \ee
Let moreover
 $$
 E_{2,\eps}:=\sum_{i,j=1}^d E^{(ij)}_\eps(t), \qquad F_{2,\eps}:=\sum_{i,j=1}^d F^{(ij)}_\eps.
 $$
For $t\in (t_\eps^-,t_\eps^+)$, we then find, combining \fref{eneg3} and \fref{Reps}, after using the Young inequality,
\bee
\frac{d E_{2,\eps}(t)}{ dt}+\eta_\eps(t)\frac{d F_{2,\eps}(t)}{ dt}&\leq & C (1+\eta_\eps(t)) E_{2,\eps}+C (1+\eta_\eps(t)) M_{1,\eps}^2.
\eee
This provides us with the estimate, for $t\in (t_\eps^-,t_\eps^+)$,
$$
E_{2,\eps}(t)+\eta_{\eps,0}F_{2,\eps}(t) \leq C \big( E_{2,\eps}(t_\eps^-)+\eta_{\eps,0}F_{2,\eps}(t_\eps^-)+M_{1,\eps}^2\big),
$$
leading to
\be \label{estF2}
F_{2,\eps}(t) \leq C \big( \eps \eta_0^{-1}E_{2,\eps}(t_\eps^-)+F_{2,\eps}(t_\eps^-)+ \eps \eta_0^{-1} M_{1,\eps}^2\big)=:(M_{4,\eps})^2.
\ee

\paragraph{Step 2.} We extend finally \fref{estF2} to $\Rm^d$, and simply need  for this to estimate $S_\eps$ in $H^{-1}(\Rm^d)$. We find
\be \label{Reps2}
 \|S_\eps\|_{H^{-1}} \leq C (1+\eta_\eps)(\|u_\eps\|_{L^2}+\| \Delta u_\eps\|_{L^2}).
 \ee
 Also,
 \be \label{vij}
\| L_1 v^{(ij)} \|_{H^{-1}} \leq C \| b \chi \nabla v^{(ij)}\|_{L^2} \leq C M_{4,\eps}. 
\ee
Combining \fref{Reps2}, \fref{vij}  with Lemma \ref{extend}, and summing over $i,j$, we obtain for $t\in (t_\eps^-,t_\eps^+)$:
\begin{align*}
  \|\Delta u_\eps(t)\|_{L^2}+\|\partial_t \Delta u_\eps(t)\|_{H^{-1}} \leq & \quad C \|\Delta u_\eps(t_\eps^-)\|_{L^2}+ C \|\partial_t \Delta u_\eps(t_\eps^-)\|_{H^{-1}}\\
  &+ C \int_{t_\eps^-}^t (1+\eta_\eps(s)) (M_{1,\eps}+M_{4,\eps}+\| \Delta u_\eps(s) \|_{L^2} )ds.
\end{align*}
Gronwall's  Lemma then yields the estimate, $t\in (t_\eps^-,t_\eps^+)$:
\begin{align} \nonumber
  \|\Delta u_\eps(t)\|_{L^2}+\|\partial_t \Delta u_\eps(t)\|_{H^{-1}} &\\
  \leq C \big(\|\Delta u_\eps(t_\eps^-)\|_{L^2}+&\|\partial_t \Delta u_\eps(t_\eps^-)\|_{H^{-1}}+ M_{1,\eps}+M_{4,\eps} \big).\label{esM5}
\end{align}

\subsection{Conclusion} We have everything needed now to conclude. Collecting \fref{M1}-\fref{Em4}-\fref{esM5}, we find, for  $t\in (t_\eps^-,t_\eps^+)$,
\be \label{fin1}
\|u_\eps(t)\|_{H^2}+\|\partial_t  u_\eps(t)\|_{H^{1}} \leq C \|u_\eps(t_\eps^-)\|_{H^3}+ C \|\partial_t  u_\eps(t_\eps^-)\|_{H^{2}}.
\ee
For $t \leq t_\eps^-$, we have $\eta_\eps(t)=0$ and perturbation has not occured yet. Since, on the one hand, $u_\eps$ and $\partial_t u_\eps$ are continuous in time along with all of their spatial derivatives, and on the other, \fref{WE} propagates the regularity of the initial conditions, we have
\be \label{fin2}
 \|u_\eps(t_\eps^-)\|_{H^3}+ C \|\partial_t  u_\eps(t_\eps^-)\|_{H^{2}} \leq  C\|u_0\|_{H^3}+ C \|u_1\|_{H^{2}}.
\ee
In the same way, we have after the perturbation, for any $t\in (t_\eps^+,\tau)$,
$$
 \|u_\eps(t)\|_{H^2}+ C \|\partial_t  u_\eps(t)\|_{H^{1}} \leq  C_\tau\|u_\eps(t_\eps^+)\|_{H^2}+ C_\tau \|u_1(t_\eps^+)\|_{H^{1}}.
$$
Together with \fref{fin1}-\fref{fin2}, this gives, for any $\tau>0$,
$$
 \|u_\eps\|_{L^\infty(0,\tau,H^2)}+ C \|\partial_t  u_\eps\|_{L^\infty(0,\tau,H^{1})} \leq  C_\tau\|u_0\|_{H^3}+ C_\tau \|u_1\|_{H^{2}}.
 $$
 This proves estimate \fref{unif} for smooth $u_0$, $u_1$. For $u_0$, $u_1$ with the regularity given in \fref{IC}, it suffices to proceed by density and a limiting argument. This concludes the proof of \fref{unif}. We now turn to the remainder term $\bR_\eps$.
 
\paragraph{Estimates on $\bR_\eps$.} Denote by $w_\eps$ the last component of the vector $\bR_\eps$. From the definition of $\bR_\eps$, we verify that is satisfies
$$
\partial_t^2 w_\eps  = a(x) \nabla \cdot \big(b(x) \nabla w_\eps \big)+a(x) \nabla \cdot \big(b(x) \eta_\eps(t) \nabla (u_\eps-U) \big), \qquad (t,x)\quad \textrm{on} \quad \Rm_+ \times \Rm^d,
$$
equipped with vanishing initial conditions. We recall that $u_\eps$ and $U$ denote the perturbed and unperturbed solutions, respectively. Given that $u_\eps$ and $U$ belong to $L^\infty_{\rm{loc}}(\Rm_+, H^3(\Rm^d))$, it follows that the above equation admit a unique solution in $L^\infty_{\rm{loc}}(\Rm_+, H^3(\Rm^d))$ with second order time derivatives in $L^\infty_{\rm{loc}}(\Rm_+, H^1(\Rm^d))$. Applying Lemma \ref{extend}, we obtain, for $t \geq T+\eps/2$,
\be \label{Ew}
  \|w_\eps(t)\|_{L^2}+\|L_0^{-1/2} (a^{-1} \partial_t w_\eps(t)) \|_{L^2} \leq C \int_0^t \eta_\eps(s) \|\nabla (u_\eps-U)(s) \|_{L^2} ds.
  \ee
  This last term is equal to
  $$
  \eta_0 \eps^{-1} \int_{-\eps/2}^{\eps/2} \|\nabla (u_\eps-U)(T+s) \|_{L^2} ds,
  $$
  which we now estimate. 
  Since $u_\eps(t)=U(t)$ for $t\leq T-\eps/2$, we have, for $s \in [-\eps/2,\eps/2]$,
  $$
  \nabla (u_\eps-U)(T+s)=\nabla u_\eps (T+s)-\nabla u_\eps(T+s-\eps)+\nabla U(T+s-\eps)-\nabla U (T+s).
  $$
  Now, we know that $\partial_t \nabla U \in L^\infty_{\rm{loc}}(\Rm_+,L^2(\Rm^d))$, and that $\partial_t \nabla u_\eps$ is uniformly bounded in  $L^\infty(0,\tau,L^2(\Rm^d))$ for all $\tau>0$ according to estimate \fref{unif}. This yields
  \be \label{EE1}
  \|\nabla (u_\eps-U)(T+s)\|_{L^2} \leq C \eps,
  \ee
  and as a consequence, together with \fref{Ew}
  \be \label{Eww}
  \|w_\eps(t)\|_{L^2}+\|L_0^{-1/2} (a^{-1} \partial_t w_\eps(t)) \|_{L^2} \leq C \eps,
  \ee
  for all $t>0$. This gives an estimate for $w_\eps$.

  We turn now to the remaining $d$ components of $\bR_\eps$, denoted $\bW_\eps$.
We verify that $\bW_\eps$ solves
  $$
  b^{-1}\frac{\partial \bW_\eps }{\partial t} + \nabla w_\eps=\chi \eta_\eps \nabla (u_\eps-U),
  $$
  with $\bW_\eps(t=0,x)=0$. Combining \fref{EE1} and \fref{Eww}, we find, for all $t>0$,
  $$
  \|\bW_\eps(t)\|_{H^{-1}}\leq C \int_0^t \big(\| \nabla w_\eps(s) \|_{H^{-1}}+\eta_\eps(s) \| \nabla (U-u_\eps)(s) \|_{H^{-1}} \big) ds \leq C_t\eps.
  $$

  This concludes the proof of Theorem \ref{mainth}.
  
\section{Proof of Theorem \ref{mainth2}} \label{secproof2}

We gather first standard compactness results. Estimate \fref{unif} leads to the existence of $u \in L^\infty_{\rm{loc}}(\Rm_+,H^2(\Rm^d))$ with $\partial_t u \in L^\infty_{\rm{loc}}(\Rm_+,H^1(\Rm^d))$, and of a subsequence $\{u_{\eps_j}\}_{j \in \NN}$ converging to $u$ in $L^\infty_{\rm{loc}}(\Rm_+,H^2(\Rm^d))$ weak-$*$ such that $\partial_t u_{\eps_j} \to \partial_t u$ in $L^\infty_{\rm{loc}}(\Rm_+,H^1(\Rm^d))$ weak-$*$ as $j \to \infty$. A consequence of such regularity for $u$ is that $u \in C^0(\Rm_+,H^1(\Rm^d))$. 

We pass now to the limit in \fref{WE}. Let $\varphi \in C^{\infty}_c(\Rm_+ \times \Rm^d)$. Then
\be \label{weak}
\langle a^{-1}\partial_t u_{\eps_j}, \partial_t \varphi \rangle= - \langle u_{\eps_j}, \nabla  \cdot [b(1+\chi\eta_{\eps_j}\big) \nabla \varphi ] \rangle,
\ee
where
$$
\langle u, v \rangle = \int_{\Rm_+} \int_{\Rm^d} u(t,x) v(t,x) dtdx.
$$
Let also
$$
\langle u, v \rangle_0 = \int_{\Rm_+}\langle u(t,\cdot), v(t,\cdot) \rangle_{H^{-1},H^1} dt.
$$
Since $\eta_{\eps_j}(t) \to \eta_0 \delta(t-T)$ as $j \to \infty$ in the sense of measures, and estimate \fref{unif} implies that $u_{\eps_j} \to u$ strongly in $C^0([0,\tau],L^2(\Omega))$ for any $\tau>0$ and any open bounded set $\Omega$, we deduce from \fref{weak} and the various convergences of $\{u_{\eps_j}\}_{j \in \NN}$ stated above that
\bea \nonumber
\langle a^{-1}\partial_t u, \partial_t \varphi \rangle_0&=&- \langle u , \nabla  \cdot  [b  \nabla \varphi ] \rangle-\eta_0 (u(T,\cdot) , \nabla  \cdot [b \chi  \nabla \varphi(T,\cdot) ] )\\
&=&\langle \nabla u, b \nabla \varphi \rangle+\eta_0 ( \nabla u(T,\cdot), b \chi \nabla \varphi(T,\cdot)).  \label{dis}
\eea
Since $a>0$ is smooth and bounded below and above, this shows in particular that
$$
|\langle \partial_t u, \partial_t \varphi\rangle_0| \leq C \|u\|_{C^0([0,\tau],H^1(\Rm^d))}\|\varphi\|_{C^0(\Rm_+,H^1(\Rm^d))}, \qquad \forall \varphi \in C_c^\infty(\Rm_+\times \Rm^d),
$$
and as a consequence $\partial_t u \in BV_{\rm{loc}}(\Rm_+,H^{-1}(\Rm^d))$. Above, $\tau$ is such that the support in time of $\varphi$ is included in $[0,\tau]$.

The next step is to identify the jump condition at $t=T$. For $k \geq 1$, consider $\varphi_k(t,x)=\chi_k(t)\psi(x)$, with $\psi \in H^1(\Rm^d)$ and $\chi_k$ the continuous function equal to one in $[T-1/k, T+1/k]$, to zero in $[0,T-2/k] \cup [T+2/k,\Rm_+)$, and that is linear in $[T-2/k,T-1/k] \cup [T+1/k,T+2/k]$. Equation \fref{dis} holds for $\varphi=\varphi_k$. The first term in the r.h.s of \fref{dis} with $\varphi=\varphi_k$ goes to zero as $k \to \infty$ by dominated convergence since $\chi_k(t) \to 0$ pointwise for $t \neq T$. The second term in the r.h.s is simply
$$
\eta_0 (\nabla u(T,\cdot), b \chi \nabla \psi).
$$
The term in the l.h.s of \fref{dis} reads
$$
k \int_{T-2/k}^{T-1/k} \langle \partial_t u(s,\cdot),\psi \rangle_{H^{-1},H^1} ds-k \int_{T+1/k}^{T+2/k} \langle \partial_t u(s,\cdot),\psi \rangle_{H^{-1},H^1} ds.
$$
Since $\partial_t u \in BV_{\rm{loc}}(\Rm_+,H^{-1}(\Rm^d))$, we can take the limit $k\to \infty$ above and obtain the expression
$$
\langle a^{-1} \partial_t u(T^-,\cdot),\psi \rangle_{H^{-1},H^1}-\langle a^{-1} \partial_t u(T^+,\cdot),\psi \rangle_{H^{-1},H^1}.
$$
Collecting results, we find
$$
\langle a^{-1}\partial_t u(T^-,\cdot),\psi \rangle_{H^{-1},H^1}-\langle a^{-1}\partial_t u(T^+,\cdot),\psi \rangle_{H^{-1},H^1}=\eta_0 ( \nabla u(T,\cdot), b \chi \nabla \psi),
$$
which is
$$
\partial_t u(T^+,x)= \partial_t u(T^-,x)+\eta_0 a(x) \nabla \cdot [b(x) \chi(x) \nabla u(T,x)] \qquad \textrm{in} \quad H^{-1}(\Rm^d),
$$
and we have obtained the jump condition.

We now show that the wave equation is satisfied almost everywhere. Taking a test function $\varphi$ in \fref{dis} with time support in $(0,T)$, we find from \fref{dis}, since $\partial_t u \in L^\infty_{\rm{loc}}(\Rm_+, H^1(\Rm^1))$,
$$
\langle a^{-1}\partial_t u, \partial_t \varphi \rangle= \langle \nabla u, b \nabla \varphi \rangle.
$$
Since $u \in L^\infty_{\rm{loc}}(\Rm_+,H^2(\Rm^d))$, this shows that $\partial^2_t u \in L^\infty(0,T,H^2(\Rm^d))$. Proceeding in the same way, we find $\partial^2_t u \in L^\infty_{\rm{loc}}(T,+\infty,H^2(\Rm^d))$. As a consequence, the equation
$$
\partial_t^2 u= a \nabla \cdot (b \nabla u)
$$
is satisfied almost everywhere in $(0,T) \times \Rm^d$ and in $(T,\Rm_+) \times \Rm^d$.

We turn finally to the uniqueness of solutions, which is straightforward: by standard uniqueness results for the wave equation, the solution $u$ is unique up to any time $t<T$. In particular, $\partial_t u$ and $\nabla u$ are uniquely defined up to $t<T$. Since $\partial_t u \in BV_{\rm{loc}}(\Rm_+,H^{-1}(\Rm^d))$, it follows that $\partial_t u(t,x)$ admits a limit in $H^{-1}(\Rm^d)$ as $t\to T^-$, and since $\nabla u$ is continuous in time with values in $H^1(\Rm^d)$ as obtained in the beginning of the proof, this shows that the term
$$
\partial_t u(T^-,x)+\eta_0 a(x) \nabla \cdot [b(x) \chi(x) \nabla u(T,x)]
$$
is uniquely defined in $H^{-1}(\Rm^d)$. Standard uniqueness for the wave equation for times $t>T$ with initial conditions $u(T^+,x)=u(T,x)$ and $\partial_t u (T^+,x)$ equal to the expression above yield finally a unique solution for all times $t \in \Rm_+$.

As a conclusion, the entire sequence $\{u_\eps\}$ converges to $u$ since \fref{WE2}-\fref{jump} admits a unique solution. This ends the proof.

\section{Appendix}

We describe here how the wave equation \fref{WE} with time-dependent coefficients arises in applications.

\paragraph{Electromagnetics.} In absence of free charges, the three dimensional Maxwell equations read
$$
\left\{
\begin{array}{l}
\partial_t \bB= - \nabla \times \bE, \qquad \nabla \cdot \bB=0\\
\partial_t \bD= \nabla \times \bH, \qquad \nabla \cdot \bD=0,
\end{array}
\right.
$$
augmented with the constitutive relations
$$
\bD(t,\bx)=\epsilon(t,\bx) \bE(t,\bx), \qquad \bB(t,\bx)=\mu(t,\bx) \bH(t,\bx),
$$
where $\bx=(x,y,z)$. The coefficients $\eps(t,x)$ and $\mu(t,x)$, supposed to be scalars here, are the permittivity and permeability of the underlying isotropic material of propagation, respectively. It is assumed that the dispersive effects of the material are neglected, and therefore that the relationship between $\bD$ and $\bE$, as well as between $\bB$ and $\bH$ is local in time. It is explained in \cite{lurie2007introduction} how time-dependent $\eps(t,x)$ and $\mu(t,x)$ can be engineered in applications.

Consider first the following transverse magnetic case where
$$
\bE(t,x,y)=
\begin{pmatrix}
 0 \\
 0 \\
  E_z(t,x,y)
\end{pmatrix},
\qquad
\bB(t,x,y)=
\begin{pmatrix}
  B_x(t,x,y)\\
  B_y(t,x,y)\\
  0
\end{pmatrix}
$$
Maxwell's equations then reduce to
$$
\left\{
  \begin{array}{l}
    \partial_t B_x = -\partial_y E_z\\
\partial_t B_y = \partial_x E_z\\
\partial_t \big(\eps E_z \big)= \partial_x \big(\mu^{-1} B_y\big)- \partial_y \big(\mu^{-1} B_x\big).
\end{array}
\right.
$$
According to the divergence free condition on $\bB$, there exists a $u$ such that $B_x=\partial_y u,$ and  $B_y=-\partial_y u$. Setting $E_z=\partial_t u$, we obtain the wave equation
$$
\partial_t \big( \epsilon (t,x) \partial_t u\big)= \nabla \cdot (\mu^{-1}(t,x) \nabla u).
$$
In the context of time reversal and ITM, it is shown in \cite{BFP} that there is no refocusing when both coefficients $\epsilon (t,x)$ and $\mu(t,x)$ are time singular. This is because the wave equation does not admit smooth solutions with respect to the time variable in this case, and the quality of refocusing can be quantified in terms of $\partial_t u$ as we have seen in Theorem \ref{mainth}. Supposing therefore that $\epsilon(t,x)$ is independent of time, we recover \fref{WE}. 

In the transverse electric case, we have,
$$
\bB(t,x,y)=
\begin{pmatrix}
 0 \\
 0 \\
  B_z(t,x,y)
\end{pmatrix}
\qquad
\bD(t,x,y)=
\begin{pmatrix}
  D_x(t,x,y)\\
  D_y(t,x,y)\\
  0
\end{pmatrix}
$$
and obtain the wave equation
$$
\partial_t \big( \mu (t,x) \partial_t u\big)= \nabla \cdot (\epsilon^{-1}(t,x) \nabla u),
$$
where $B_z=\partial_t u$, $D_x=-\partial_y u$,  and $D_y=\partial_y u$. We recover now \fref{WE} when the permeability $\mu$ is independent of time.

\paragraph{Elasticity.} In the case of an isotropic material with negligible shear modulus (i.e. a non-rigid material), the Navier-Cauchy equations reduce to
$$
\partial_t \big(\rho(t,x) \partial_t \bu \big)= \nabla \big( \lambda(t,x) \nabla \cdot \bu\big),
$$
where $\bu \in \Rm^3$ is the displacement field, $\rho$ is the mass density, and $\lambda$ the first Lam\'e coefficient. Again, one can find in \cite{lurie2007introduction} examples of mechanical systems with time-dependent $\rho$ and $\lambda$.

Set $\bw= \rho \partial_t \bu$. If $\bw$ is irrotational at $t=0$, then $\bw$ remains irrotational because of the above equation, and we can write $\bw(t, \bx)= \nabla \phi(t,\bx)$. With $p=- \lambda \nabla \cdot \bu$, we obtain the system
$$
\left\{
\begin{array}{l}
\partial_t \bw= \nabla p \\
\partial_t \big(\lambda^{-1} p \big)= -\nabla \big(\rho^{-1} \bw \big).
\end{array}
\right.
$$
Defining $p=-\partial_t \phi$, we find the wave equation
$$
\partial_t \big(\lambda^{-1} \partial_t \phi \big)= \nabla \big(\rho^{-1} \nabla \phi \big).
$$
We then recover \fref{WE} when $\lambda^{-1}$ does not depend on time.

\paragraph{Fluids.} In the seminal work \cite{Fink-nphys}, waves at the surface of a water tank are modeled by the following wave equation, neglecting surface tension and therefore dispersive effects,

$$
\partial^2_t u= c_0^2 (1+\alpha \delta(t-T)) \Delta u, \qquad \textrm{in } \Rm^2. 
$$
Above, $u$ is the surface height, $c_0$ is the (constant) background velocity, and $\alpha \delta(t-T)$ represents the action of the ITM at $t=T$. This wave equation can be derived from the Euler system, see e.g. \cite{ursell}.

Note that the wave equation \fref{WE} with time singular coefficients seems more difficult to justify physically in the context of sound waves. Indeed, while it is possible to derive some type of wave equations with time-dependent coefficients for sound waves, see \cite{MARTIN}, these models are justified for slowly varying (in time) coefficients, which is certainly not the case for an ITM. 

  \bibliographystyle{plain}
\bibliography{bibliography.bib}
\end{document}